\documentclass{amsart}

\usepackage{amsthm,amssymb,amsmath}
\usepackage{mdwlist} 

\newtheorem{theorem}{Theorem}

\newtheorem{corollary}[theorem]{Corollary}

\theoremstyle{definition}

\newtheorem{example}[theorem]{Example}

\newtheorem*{rem}{Remark} 
\newtheorem{remark}[theorem]{Remark}

\newcommand{\C}{\mathbb{C}}
\newcommand{\R}{\mathbb{R}}
\newcommand{\D}{\mathbb{D}}
\newcommand{\N}{\mathbb{N}}
\newcommand{\dd}{\, d}
\newcommand{\ip}[2]{\left\langle {#1 ,#2} \right\rangle}
\newcommand{\norm}[1]{\left \lVert#1 \right \rVert}
\DeclareMathOperator{\dist}{dist}

\title[Weighted Bergman Kernel Functions]{Weighted Bergman Kernel Functions\\
Associated to Meromorphic Functions}
\author{Robert Jacobson}

\date{\today}

\address{Roger Williams University\\ One Old Ferry Road\\ Bristol, RI 02809}
\email{rljacobson@member.ams.org}
\thanks{This work was partially supported by a grant from The Foundation To
Promote Scholarship \& Teaching at Roger Williams University.}

\begin{document}
\begin{abstract}
We present a technique for computing explicit, concrete formulas for the
weighted Bergman kernel on a planar domain with weight the modulus squared of a
meromorphic function in the case that the meromorphic function has a finite
number of zeros on the domain and a concrete formula for the unweighted kernel
is known. We apply this theory to the study of the Lu Qi-keng Problem.
\end{abstract}
\maketitle

\section{Introduction}
The Bergman kernel function has been called a cornerstone of geometric
function theory~\cite{krantz_new_2006} and is an object of considerable study
in complex analysis. The problems of computing explicit formulas for this function and
determining its zero set are classical problems in complex analysis. Domains for
which the associated Bergman kernel is zero-free are called Lu {Qi-keng}
domains, and the problem of determining which domains are Lu {Qi-keng} is known
as the Lu {Qi-keng} Problem. This problem is of interest in the study of Bergman
representative coordinates which require the kernel to be zero-free
(see~\cite{jarnicki_invariant_1993, jarnicki_invariant_2005}). The Lu {Qi-keng}
Problem for smooth planar domains has been solved~\cite{suita_lu_1976}, but a
solution for higher dimensions is not yet known~\cite{boas_lu_2000}. The
property of having a zero-free kernel is also a biholomorphic invariant and
hence may be used to distinguish biholomorphic equivalence classes.

The main result of this paper is Theorem~\ref{decomp:ZeroReducingTheorem}, which
allows one to write certain weighted Bergman kernels on the plane in terms of
other weighted Bergman kernels with simpler weights. One consequence of this
theorem is that if one has an explicit, concrete formula for an unweighted
kernel, then one can compute an explicit, concrete formula for the weighted
kernel whenever the weight is the modulus squared of a meromorphic function with
finitely many zeros on the associated domain. By the well-known technique of
Theorem~\ref{prelim:Forelli-Rudin}, weighted kernels for domains on the plane
are related to unweighted kernels for domains in $\C^2$. Thus, the results
presented here that are specific to complex dimension $1$ have relevance to the
classical problems of the first paragraph, in particular the Lu {Qi-keng}
Problem, in complex dimension $2$. With the $2$-dimensional Lu {Qi-keng} Problem
in mind, we study the zero sets of weighted kernels in Section~\ref{sect:zeros}.

The Bergman kernel for a domain $\Omega\subset\C$ is
the unique skew-symmetric sesqui-holomorphic\footnote{Sesqui-holomorphic means
holomorphic in the first variable and conjugate holomorphic in the second
variable.} function $K^\Omega\colon \Omega\times \Omega \to \C$ with the
reproducing property
\begin{align}
f(z)=\ip{f}{K^\Omega(\cdot, z)} = \int_\Omega f(w) K^\Omega(z,w) \dd V_w \;\;\;
\mbox{ for all } f \in A^2(\Omega),
\label{intro:reprokernel}
\end{align}
where $dV_w$ is the real $2n$-dimensional Lebesgue volume (or area) measure,
and $A^2(\Omega)$ is the Hilbert space of square-integrable holomorphic
functions on $\Omega$, called the Bergman space. (When the domain is clear, we
will omit it from the superscript of $K$.) Equivalently, if $\{\phi_j
\}_{j=0}^\infty$ is an orthonormal Hilbert space basis for $A^2(\Omega)$, then the Bergman kernel
function $K^\Omega(z,w)$ is given by
\begin{equation}
K^\Omega(z, w):=\sum_{j=0}^\infty \phi_j(z) \overline{\phi_j(w)}. \label{intro:bkdef}
\end{equation}

Also of present interest is the weighted Bergman kernel with respect to a
nonnegative weight function $\varphi$, which we denote $K^\Omega_\varphi(z,w)$.
Replacing the inner product in~(\ref{intro:reprokernel}) with the weighted inner
product
\[
\ip{f}{g}_\varphi := \int_\Omega f(w) \overline{g(w)} \varphi(w)\dd V_w,
\]
the kernel $K^\Omega_\varphi(z,w)$ is the unique reproducing kernel for the
weighted Bergman space $A^2_\varphi(\Omega)=\{f \mid \ip{f}{f}_\varphi<\infty$
and $f$ holomorphic$\}$.

The details of this classical theory can be found in many texts on complex
analysis, for example in~\cite{bell_cauchy_1992,krantz_function_2001}.

\section{Preliminary theory}\label{prelim:section}

To study the Lu {Qi-keng} Problem in higher dimensions, we would like concrete
examples of kernels on domains in $n$-dimensional complex space, but obtaining a
closed-form formula for the kernel from \eqref{intro:bkdef} is possible only for
domains with a high degree of symmetry. There are, however, several techniques
for relating the kernel of one domain to the kernel of another domain of
different complex dimension (see~\cite{boas_bergman_1999}). We shall make
crucial use of the following known result.

\begin{theorem}\label{prelim:Forelli-Rudin}
Let $D$ be a bounded domain in $\C$, let $\varphi(z) \colon D \to [0, \infty]$
be a weight function on $D$, and let $\Omega$ be defined
by $\Omega:=\{(z, w)\in \C^2 \mid z \in D, |w| < \varphi(z)\} \subset
\C^2$. Then $K_{\pi\varphi^2}(z, w) \equiv K(z, 0, w, 0)$.
\end{theorem}

The idea behind this theorem appears in the literature in various forms.
Theorem~\ref{prelim:Forelli-Rudin} is essentially Corollary~2.1
of~\cite{ligocka_forelli-rudin_1989} which Ligocka, generalizing an idea found
in a proof due to Forelli and Rudin in~\cite{forelli_projections_1974}, calls
the Forelli{\textendash}Rudin construction. The term Forelli{\textendash}Rudin
construction appears elsewhere in subsequent literature in reference to similar
techniques. Such techniques are surveyed in~\cite{boas_bergman_1999}.

Our primary goal is to express a weighted kernel in terms of another weighted
kernel that is in some sense simpler than the first. The following theorem is
the simplest case of such a theorem and is fundamental to the rest of the
theory.

\begin{theorem}\label{prelim:gTheorem}
Let $\Omega \subset \C^n$, let $K_\varphi(z, w)$ be the weighted Bergman kernel
on $\Omega$ with respect to a weight function $\varphi$, and let $g$ be
meromorphic on $\Omega$. Suppose that, after possibly removing singularities,
$\frac{K_\varphi(z, w)}{g(z)}$ is holomorphic in $z$. Then $K_{\varphi \cdot
|g|^2}(z, w) = \frac{K_\varphi(z, w)}{g(z) \overline{g(w)}}$.
\end{theorem}

\begin{proof}
The weight $\varphi$ plays no role in the following argument, so for simplicity
of notation, we suppress the subscript $\varphi$ in the calculation. We have
that
\[
\int_\Omega \left| \frac{K(z, w)}{g(z) } \right|^2 |g(z)|^2 \dd V_z =
\norm{K(\cdot, w)}^2 < \infty,
\] 
so
\begin{equation}\label{prelim:gTheoremEqi}
\frac{K(z, w)}{g(z)} \in A^2_{|g|^2}(\Omega) \text{ (as a
function of }z).
\end{equation}
Also, 
\[
\int_\Omega | K_{|g|^2}(z, w)|^2 |g(z)|^2 \dd V_z =
\norm{K_{|g|^2}(\cdot, w)}^2_{|g|^2}<\infty,
\]
so
\begin{equation}\label{prelim:gTheoremEqii}
K_{|g|^2}(z, w)g(z) \in A^2(\Omega) \text{ (as a function
of }z).
\end{equation}
By~\eqref{prelim:gTheoremEqi} and the reproducing property of $K_{|g|^2}(z, w)$,
we have
\begin{align*}
\frac{K(z,w)}{g(z)} &= \int_\Omega
\frac{K(\zeta,w)}{g(\zeta)} K_{|g|^2}(z, \zeta) |g(\zeta)|^2 \dd V_\zeta \\
&= \int_\Omega K(\zeta, w) K_{|g|^2}(z, \zeta)
\overline{g(\zeta)} \dd V_\zeta \\
&= \overline{ 
\int_\Omega K(w, \zeta) K_{|g|^2}( \zeta, z)g(\zeta) \dd V_\zeta 
}.
\end{align*}
By~\eqref{prelim:gTheoremEqii} and the reproducing property of the kernel $K(z,
w)$, this last expession is
$
\overline{K_{|g|^2}(w, z) g(w)} = \overline{g(w)}K_{|g|^2}(z, w).
$
We have shown that $\frac{K(z,w)}{g(z)}=\overline{g(w)}K_{|g|^2}(z, w)$, from
which the theorem follows.
\end{proof}

This theorem  and the ancillary Theorem~\ref{zeros:outerconeweighted} are the
only multidimensional theorems in this paper. The other results are specific to
domains of dimension $1$. As described above, the $1$-dimensional results
together with Theorem~\ref{prelim:Forelli-Rudin} can be used to study domains in
higher dimensions. Indeed, Theorem~\ref{prelim:gTheorem} provides a recipe,
illustrated by Example~\ref{prelim:gTheoremExample}, for constructing non Lu
{Qi-keng} domains in $\C^2$. The technique of this example, though elementary,
appears to be absent from the literature.

\begin{example}\label{prelim:gTheoremExample}
Let $c$ be a point in the open unit disk $\D$, and define
$\varphi(z):=|z-c|^{-2}$ and $\Omega := \left\{ (z, w)\in \C^2 \mid z \in \D,
|w| < \sqrt{\frac{1}{\pi} \varphi(z)} \right\} \subset \C^2$. By
Theorem~\ref{prelim:Forelli-Rudin}, the kernel $K_{\varphi}^\D(z, w)$ satisfies
$K_{\varphi}^\D(z, w) \equiv K^\Omega(z, 0, w, 0)$.
Theorem~\ref{prelim:gTheorem} gives $K_{\varphi}^\D(z, w) = (z-c)K^{\D}(z,
w)(\overline{w} - \overline{c})$, which clearly has zeros whenever $z=c$ or
$w=c$. Hence $\Omega$ is not Lu {Qi-keng}.
\end{example}

The domain in Example~\ref{prelim:gTheoremExample} is an unbounded domain, but
a bounded non Lu {Qi-keng} domain can be obtained via Ramanadov's Theorem
together with Hurwitz's Theorem.

Our goal is now to obtain a formula for a weighted kernel explicitly in terms of
the unweighted kernel when the weight is the modulus squared of a meromorphic
function. Theorem~\ref{prelim:gTheorem} allows us to handle the poles: the poles
appear as zeros of the same order in the formula for the weighted kernel given
by Theorem~\ref{prelim:gTheorem}. On the other hand, any zeros of the
meromorphic function associated to the weight clearly cannot appear as poles in
the formula for the weighted kernel since the kernel is holomorphic.

\section{Decomposition theorems}\label{decomp:section}

Theorem~\ref{prelim:gTheorem} needs modification in the case that the
meromorphic function in the weight vanishes. The goal of this section is to show
how to accomplish this modification in dimension $1$. For a general planar
domain $\Omega$ and holomorphic function $f$, we are able to express
$K^\Omega_{|f|^2}(z, w)$ in terms of the kernel associated to a ``simpler''
weight function and the basis functions for the orthogonal complement of
$A^2_{|f|^2}(\Omega)$ in a larger space of functions; when $\varphi$ is both
bounded and bounded away from zero near $c$, the normalized function
$\frac{K^{\Omega}_\varphi(z,c)}{(z-c)\sqrt{K^{\Omega}_\varphi(c,c)}}$ turns out
to span the orthogonal complement of $A^2_{|z-c|^2\varphi(z)}(\Omega)$ in
$A^2_{|z-c|^2\varphi(z)}(\Omega \setminus \{c\})$.

\begin{theorem}\label{decomp:ZeroReducingTheorem}
Let $\Omega\subset \C$ be a domain, $c\in\Omega$, and $\varphi$ be a
weight on $\Omega$ which is bounded in a neighborhood of $c$. Then
\begin{equation}
K^\Omega_{|z-c|^2\varphi}(z, w) = \frac{K^{\Omega}_\varphi(z,
w)}{(z-c)(\overline{w}-\overline{c})} -
\frac{K^{\Omega}_\varphi(z,c)K^{\Omega}_\varphi(c,
w)}{(z-c)(\overline{w}-\overline{c})K^{\Omega}_\varphi(c,c)}.
\label{decomp:ZeroReducingFormula}
\end{equation} 
\end{theorem}

\begin{rem}
The requirement that $\varphi$ be bounded in a neighborhood of $c$ excludes
degenerate cases such as $\varphi(z)=\frac{1}{|z|^2}$ with $c=0$. The right hand
side of Equation~\ref{decomp:ZeroReducingFormula} apparently has singularities
at $z=c$ and $w=c$, but these singularities are removable.
\end{rem}

\begin{proof}
Let $\psi(z):=\frac{K^\Omega_\varphi(z, c)}{z-c}$. Clearly $\psi\in
A^2_{|z-c|^2\varphi}(\Omega \setminus \{c\})$. Our strategy is as follows:
\begin{enumerate}
  \setlength{\itemsep}{\baselineskip}
  \item\label{decomp:Thing1} $\frac{K^{\Omega}_{\varphi}(z,
  w)}{(z-c)(\overline{w}-\overline{c})}$ reproduces elements of
  $A^2_{|z-c|^2\varphi}(\Omega)$ in $A^2_{|z-c|^2\varphi}(\Omega\setminus
  \{c\})$.
  \item\label{decomp:Thing2} $\psi(z)$ is orthogonal to
  $A^2_{|z-c|^2\varphi}(\Omega)$ in $A^2_{|z-c|^2\varphi}(\Omega\setminus
  \{c\})$; as a consequence, 
  \item\label{decomp:Thing3} $\psi(z)$ is orthogonal to
  $K^{\Omega}_{|z-c|^2\varphi}(z, w)$ in $A^2_{|z-c|^2\varphi}(\Omega\setminus
  \{c\})$.
  \item\label{decomp:Thing4} From \eqref{decomp:Thing1} and
  \eqref{decomp:Thing2}, $Q(z, w):=\frac{K^{\Omega}_{\varphi}(z,
  w)}{(z-c)(\overline{w}-\overline{c})} - c_0(w)\psi(z)$ also reproduces
  elements of $A^2_{|z-c|^2\varphi}(\Omega)$ in
  $A^2_{|z-c|^2\varphi}(\Omega\setminus \{c\})$, where $c_0(w)$ is arbitrary.
  \item\label{decomp:Thing5} Setting
  $c_0(w):=\overline{\psi(w)}/K^\Omega_\varphi(c, c)$, we have $Q\in
  A^2_{|z-c|^2\varphi}(\Omega)$; it follows from \eqref{decomp:Thing4} and
  the uniqueness of the Bergman kernel that $Q(z, w)\equiv
  K^\Omega_{|z-c|^2\varphi}(z, w)$.
\end{enumerate} 
Once \eqref{decomp:Thing1} and \eqref{decomp:Thing2} are proven,
\eqref{decomp:Thing3} and \eqref{decomp:Thing4} are obvious.

\noindent\emph{Proof of~\eqref{decomp:Thing1}:} Let $f\in
A^2_{|z-c|^2\varphi}(\Omega)$.
We have
\begin{align*}
&\int_{\Omega\setminus \{c\}} f(w) \frac{K^{\Omega}_{\varphi}(z,
  w)}{(z-c)(\overline{w}-\overline{c})} |w-c|^2\varphi(w) \dd V_w \\
&\qquad =\frac{1}{z-c} \int_\Omega K^\Omega_\varphi (z, w)
f(w)(w-c)\varphi(w) \dd V_w \\ 
&\qquad =\frac{1}{z-c} f(z)(z-c) &(\text{since } f(z)(z-c)\in A^2_\varphi
(\Omega))\\
&\qquad =f(z).
\end{align*}
This proves~\eqref{decomp:Thing1}.

\noindent\emph{Proof of~\eqref{decomp:Thing2}:} Let $f\in
A^2_{|z-c|^2\varphi}(\Omega)$.
We have
\begin{align*}
&\int_{\Omega\setminus \{c\}} f(w) \overline{\psi(w)} |w-c|^2\varphi(w) \dd V_w
\\
&\qquad =\int_{\Omega\setminus \{c\}} f(w) \frac{\overline{K^\Omega_\varphi(w,
c)}}{\overline{w}-\overline{c}} |w-c|^2\varphi(w) \dd V_w \\
&\qquad =\int_{\Omega} f(w)(w-c) K^\Omega_\varphi(c, w)\varphi(w) \dd V_w \\
&\qquad =0 &(\text{since } f(z)(z-c)\in A^2_\varphi(\Omega)).
\end{align*}
This proves~\eqref{decomp:Thing2}.

To finish the proof, observe that
for $c_0(w):=\overline{\psi(w)}/K^\Omega_\varphi(c, c)$, we have that
\[
Q(z, w)\equiv \frac{K^{\Omega}_\varphi(z, w)}{(z - c)(\overline{w} -
\overline{c})} - \frac{K^{\Omega}_\varphi(z,c)K^{\Omega}_\varphi(c, w)}{(z -
c)(\overline{w} - \overline{c})K^{\Omega}_\varphi(c,c)},
\]
which has a removable singularity at $z=c$ and $w=c$. Thus~\eqref{decomp:Thing5}
holds, and the theorem is proven.
\end{proof}

Theorem~\ref{decomp:ZeroReducingTheorem} combined with
Theorem~\ref{prelim:gTheorem} allows one to produce an explicit formula for
$K^\Omega_{|f|^2}(z, w)$ in terms of $K^\Omega(z, w)$ in the case that $f$ is
meromorphic on $\Omega$ with a finite number of zeros by just iterating the
formula of Equation~\ref{decomp:ZeroReducingFormula}. In fact,
Theorem~\ref{decomp:ZeroReducingTheorem} is a special case of the following more
general theorem.

\begin{theorem}\label{decomp:ZeroReducingTheoremGeneral}
Let $\Omega$ be a planar domain, $\{c_j\}_{j=1}^m$ a sequence of $m$ distinct
points in $\Omega$, $\{\alpha_j\}_{j=1}^m$ a sequence of positive integers, and
$\varphi$ a weight such that for all $j$, $\varphi$ is both bounded and bounded
away from zero in a neighborhood of $c_j$. Define the following polynomials:
\begin{align*}
p(z)&:=(z-c_1)^{\alpha_1}(z-c_2)^{\alpha_2}\dotsm (z-c_m)^{\alpha_m};\\
p_{j,k}(z)&:=(z-c_1)^{\alpha_1}(z-c_2)^{\alpha_2}\dotsm
(z-c_{j-1})^{\alpha_{j-1}} (z-c_j)^{k},\;\;(1\leq j \leq m, 1\leq k \leq
\alpha_j);\\
q_{j,k}(z)&:=p(z)/p_{j,k}(z) \\
&=(z-c_{j})^{\alpha_{j-k}}(z-c_{j+1})^{\alpha_{j+1}}(z-c_{j+2})^{\alpha_{j+2}}\dotsm
(z-c_m)^{\alpha_{m}}.
\end{align*}
Then
\[
K^\Omega_{|p(z)|^{2}\varphi}(z, w) =
\frac{K^\Omega_\varphi(z, w)}{p(z)\overline{p(w)}}
-\sum_{j=1}^m \sum_{k=1}^{\alpha_j} \frac{K^\Omega_{|q_{j,k}|^{2}\varphi}(z,c_j)
K^\Omega_{|q_{j,k}|^{2}\varphi}(c_j, w)}{ p_{j,k}(z) \overline{p_{j,k}(w)}
K^\Omega_{|q_{j,k}|^{2}\varphi}(c_j, c_j)}.
\]
\end{theorem}

\begin{remark}
By the $L^2$-version of the Riemann Removable Singularity
Theorem~\cite[E.3.2]{range_holomorphic_2010}, when a weight $\psi$ is both
bounded and bounded away from zero in a neighborhood of $c$, then
$K_\psi^{\Omega \setminus \{c\}} (z, w) \equiv K_\psi^{\Omega}(z, w)$.
\end{remark}

\begin{proof}
We wish to show that the functions $\psi_{j,k}(z):=\frac{K^\Omega_{|q_{j,
k}|^2\varphi}(z, c_j)}{p_{j, k}(z)}$ form a basis for the orthogonal complement of
$A^2_{|p|^2\varphi}(\Omega)$ in $A^2_{|p|^2\varphi}(\Omega \setminus
\{c_j\}_{j=1}^m)$. We prove only that the $\psi_{j, k}$ are mutually orthogonal,
the rest of the proof being an easy exercise.

For $\psi_{j_0, k_0}$ and $\psi_{j_1, k_1}$ distinct, we may assume $j_0>j_1$ or
else $j_0=j_1$ and $k_0>k_1$. Then 
\[
p_{j_0, j_1}(z) = p_{j_1,k_1}(z) (z-c_{j_1})^{\alpha_{j_1}-k_1}
(z-c_{j_1+1})^{\alpha_{j_1+1}} \dotsm (z-c_{j_0})^{k_0},
\]
and
\begin{align*}
&\ip{\psi_{j_0, k_0}(z)}{\psi_{j_1, k_1}(z)}_{|p|^2\varphi} \\
&\qquad = \int_{\Omega \setminus
\{c_j\}_{j=1}^m} \frac{K^\Omega_{|q_{j_0, k_0}|^2\varphi}(z, c_{j_0})}{p_{j_0,
k_0}(z)}
\frac{K^\Omega_{|q_{j_1, k_1}|^2\varphi}(c_{j_1}, z)}{\overline{p_{j_1,
k_1}(z)}} |p(z)|^2\varphi(z) \dd V_z \\
&\qquad = \int_\Omega K^\Omega_{|q_{j_0, k_0}|^2\varphi}(z, c_{j_0}) \\
& \qquad\qquad \times K^\Omega_{|q_{j_1, k_1}|^2\varphi}(c_{j_1}, z) \overline{
(z-c_{j_1})^{\alpha_{j_1}-k_1} (z-c_{j_1+1})^{\alpha_{j_1+1}} \dotsm
(z-c_{j_0})^{k_0} } \\
& \qquad\qquad \times |q_{j_0, k_0}(z)|^2\varphi(z) \dd V_z \\ 
&\qquad=0.
\end{align*}
\end{proof}

The proof does not depend on $m$ being finite; we can still construct an
orthonormal basis for the orthogonal complement of $A^2_{|p|^2\varphi}(\Omega)$
in $A^2_{|p|^2\varphi}(\Omega \setminus \{c_j\}_{j=1}^m)$. However, this is of
limited practical value since in that case
Theorem~\ref{decomp:ZeroReducingTheoremGeneral} fails to give a closed form
expression for the original weighted kernel. Moreover, in practice the simpler
Theorem~\ref{decomp:ZeroReducingTheorem} is sufficient.

\section{Zeros of weighted kernels}\label{sect:zeros}

We now study the relationship of the zeros of these weighted kernels have to the
zeros of the simpler kernels.

\begin{theorem}\label{zeros:ZerosTheoremA}
Let $\Omega$ be a domain in $\C$, let $c,z_0, w_0 \in \Omega$, and let $\varphi$
be a weight on $\Omega$ that is bounded and bounded away from zero in some
neighborhood of $c$. Suppose $K_{|z-c|^2\varphi}(z_0, w_0)=0$. Then
$K_{\varphi}(z_0, w_0)=0$ if and only if either $K_{\varphi}(z_0, c)=0$ or
$K_{\varphi}(c, w_0)=0$.
\end{theorem}

\begin{proof}
By the hypothesis and Theorem~\ref{decomp:ZeroReducingTheorem}, 
\[
0 
= \frac{K_\varphi(z_0,
w_0)}{(z_0-c)(\overline{w_0}-\overline{c})} -
\frac{K_\varphi(z_0,c)K_\varphi(c,
w_0)}{(z_0-c)(\overline{w_0}-\overline{c})K_\varphi(c,c)},
\]
from which the theorem is evident.
\end{proof}

Requiring that $\varphi$ be bounded  and bounded away from zero in a
neighborhood of $c$ determines the order of the zero of the weight
$|z-c|^2\varphi(c)$ to be two, a fact to which there are two significant
consequences. First, as a consequence of the $L^2$-version of the Riemann
Removable Singularity Theorem, $K_\varphi^\Omega(z, w) \equiv K_\varphi^{\Omega
\setminus \{c\}}(z, w)$ on $(\Omega \setminus \{c\})\times (\Omega \setminus
\{c\})$. We employ this fact in the next several theorems without comment.
Second, for zeros of higher orders in the weight, we would need to use
Theorem~\ref{decomp:ZeroReducingTheoremGeneral} rather than
Theorem~\ref{decomp:ZeroReducingTheorem}, which would not give the
conclusion of Theorem~\ref{zeros:ZerosTheoremA}.

Theorem~\ref{zeros:ZerosTheoremA} says the value of
$K^\Omega_{\varphi}(z,w)$ at $c$ affects the zero set of
$K^\Omega_{|z-c|^2\varphi}(z,w)$. Compare this to the case that
$c\not\in\Omega$, in which case Theorem~\ref{prelim:gTheorem} says that the zero
sets of both kernels coincide.

Theorem~\ref{zeros:ZerosTheoremA} assumes $K^\Omega_{|z-c|^2\varphi}(z,w)$ has a
zero and then says when $K^\Omega_{\varphi}(z,w)$ has a zero. The next theorem
assumes $K_\varphi(z, w)$ has a zero and then says when
$K^\Omega_{|z-c|^2\varphi}(z,w)$ has a zero.

\begin{theorem}\label{zeros:ZerosTheoremB}
Let $\Omega$ be a domain in $\C$, let $z_0, c\in\Omega$ with $z_0\neq c$, and
let $\varphi$ be a weight on $\Omega$ that is bounded and bounded away from zero
in some neighborhood of $c$. Suppose $K_\varphi(z_0, c)=0$. Then
$K_{|z-c|^2\varphi}(z_0, w)$ has a zero of order $m-1$ at $w=c$ if and only if
$K_\varphi(z_0, w)$ has a zero of order $m$ at $w=c$.
\end{theorem}

\begin{proof}
By Theorem~\ref{decomp:ZeroReducingTheorem}, 
\begin{align*}
K_{|z-c|^2\varphi}(z_0, w) &= \frac{K_\varphi(z_0,
w)}{(z_0-c)(\overline{w}-\overline{c})} -
\frac{K_\varphi(z_0,c)K_\varphi(c,w)}{(z_0-c)
(\overline{w}-\overline{c})K_\varphi(c,c)} \\
&= \frac{1}{z_0-c}\cdot \frac{K_\varphi(z_0, w)}{\overline{w}-\overline{c}}.
\end{align*}
If $m$ is the order of the zero of $K_\varphi(z_0, w)$ at $w=c$, then this last
expression has a zero of order $m-1$ at $w=c$.
\end{proof}

\begin{theorem}\label{zeros:ZerosTheoremC}
Let $\Omega$ be a domain in $\C$, let $c_0,c_1, c_2\in\Omega$ be distinct, and
let  $\varphi$ be a weight on $\Omega$ that in some neighborhood of $c_0$ is
bounded and bounded away from zero. Suppose either $K_\varphi(c_0,
c_1) = 0$ or $K_\varphi(c_0, c_2)=0$. Then $K_{|z-c_0|^2\varphi}(c_1, c_2)=0$ if
and only if $K_{\varphi}(c_1, c_2)=0$.
\end{theorem}

\begin{proof}
By Theorem~\ref{decomp:ZeroReducingTheorem}, 
\begin{align*}
K_{|z-c_0|^2\varphi}(c_1, c_2) &= \frac{K_\varphi(c_1,
c_2)}{(c_1-c_0)(\overline{c_2}-\overline{c_0})} -
\frac{K_\varphi(c_1,c_0)K_\varphi(c_0,c_2)}
{(c_1-c_0)(\overline{c_2}-\overline{c_0})K_\varphi(c_0,c_0)}
\\ &= \frac{1}{(c_1-c_0)(\overline{c_2}-\overline{c_0})}\cdot K_\varphi(c_1,
c_2),
\end{align*}
from which the theorem is evident.
\end{proof}

\begin{theorem}\label{zeros:ZerosTheoremD}
Let $\Omega$ be a domain in $\C$, and let $\varphi$ be a weight on $\Omega$.
Suppose that for some $c_0\in\partial\Omega$ and some sequence
$\{c_j\}_{j=1}^\infty$ in $\Omega$ converging to $c_0$, we have $\frac{K_\varphi(z,\,
c_j)}{K_\varphi(c_j,\, c_j)}\to 0$ as $j\to\infty$ for all fixed $z\in\Omega$.
Suppose also that there exist $z_0, w_0 \in \Omega$ such that $K_\varphi(z_0,
w_0)=0$ and that $K_\varphi(z, c_j)$ is bounded away from $0$ when $j$ is large
enough and $z$ is in a compact subset of $\Omega$. Then for sufficiently
large $j$ (i.e., for $c_j$ sufficiently close to $c_0\in\partial \Omega$), there
exists a $z_1=z_1(c_j) \in \Omega$ near $z_0$ such that
$K_{|z-c_j|^2\varphi}(z_1, w_0)=0$.
\end{theorem}

\begin{proof}
Define the following for all $\zeta, \omega, z\in\Omega$ and $\varepsilon>0$:
\begin{align*}
g_{\zeta, \omega}(z)&:=\frac{K_\varphi(z, \omega)}{K_\varphi(z, \zeta)};\\
\alpha(\zeta)&:=|g_{\zeta, w_0}(\zeta)| 
= \left| \frac{K_\varphi(\zeta, w_0)}{K_\varphi(\zeta, \zeta)} \right|;\mbox{
and}\\
B(z, \varepsilon)&:= \{w\in\Omega \mid |z-w| < \varepsilon \} \;
\mbox{ (the usual open }\varepsilon\mbox{-ball about
}z\mbox{).}
\end{align*}
Observe that by hypothesis, $\alpha(c_j)\to 0$ as $j\to\infty$. 
Let $d:=\frac{1}{2}\dist(z_0, \partial\Omega)$. Choose a $j_0\in\N$ so that the
following hold:
\begin{enumerate}
  \item\label{zeros:jlargeenough} $\frac{1}{j_0} < d$, and
  \item $|c_j - c_0| < \frac{1}{j_0}$ for all $j>j_0$.
\suspend{enumerate}
By~\eqref{zeros:jlargeenough} and the definition of $d$, 
\resume{enumerate}
  \item the closed ball $\overline{B\left(z_0, \frac{1}{j_0}\right)}$ is contained in
  $\Omega$.
\suspend{enumerate}
By hypothesis, for $j$ large enough, $K_\varphi(z, c_j)$ is bounded away from
zero for $z\in \overline{B\left(z_0, \frac{1}{j_0}\right)}$. Thus for $j$ large
enough, the zeros of $g_{c_j, w_0}(z):= \frac{K_\varphi(z,\, w_0)}{K_\varphi(z,\, c_j)}$
correspond to the zeros of $K_\varphi(z, w_0)$ on $\overline{B\left(z_0,
\frac{1}{j_0}\right)}$. So by possibly increasing $j_0$, we can choose $j_0$
large enough so that we also have 
\resume{enumerate}
  \item $\overline{B\left(z_0, \frac{1}{j_0}\right)}$ contains a single zero of $g_{c_j,
  w_0}(z)$ when $j>j_0$, namely $z_0$.
\suspend{enumerate}
Now choose $j_1 \geq j_0$ such that
\resume{enumerate}
  \item $\alpha(c_j) < \frac{1}{j_0}$ for all $j\geq j_1$, and
  \item $\alpha(c_j) < \inf \left\{ |g_{c_{j_1}, w_0}(z)| \mid z\in \partial
  B\left(z_0, \frac{1}{j_0}\right) \right\}$ for all $j \geq j_1$.
\end{enumerate}

Now we argue that $C_0:=g_{c_{j_1}, w_0}(\partial B\left(z_0,
\frac{1}{j_0}\right))$ is a closed curve about the origin and the point
$g_{c_{j_1}, w_0}(c_{j_1})$. Since $z_0$ is a zero of the holomorphic function
$g_{c_{j_1}, w_0}(z)$ and $\partial B\left(z_0, \frac{1}{j_0}\right)$ is a
closed curve about $z_0$, it follows from the argument principle of the
elementary theory of holomorphic functions that $C_0$ is a closed curve about
the origin. Moreover, $\alpha(c_{j_1}) < \inf \left\{ |g_{c_{j_1}, w_0}(z)|
\mid z\in \partial B\left(z_0, \frac{1}{j_0}\right) \right\}$ by (6), and so
$C_0$ also encloses a region containing $g_{c_{j_1}, w_0}(c_{j_1})$, that is,
$|g_{c_{j_1}, w_0}(c_{j_1})| < |g_{c_{j_1}, w_0}(z)|$ on $\partial B\left(z_0,
\frac{1}{j_0}\right)$. By Rouch\'e's
Theorem~\cite[p.~110]{boas_invitation_2010}, it follows that the
function $g_{c_{j_1}, w_0}(z)-g_{c_{j_1}, w_0}(c_{j_1})$ has a zero in
$B\left(z_0, \frac{1}{j_0}\right)$. Hence for some $z_1\in B\left(z_0, \frac{1}{j_0}\right)$,
we have $g_{c_{j_1}, w_0}(z_1) = g_{c_{j_1}, w_0}(c_{j_1})$, which is equivalent
to $\frac{K_\varphi(z_1,\, w_0)}{K_\varphi(z_1,\, c_{j_1})}
=\frac{K_\varphi(c_{j_1},\, w_0)}{K_\varphi(c_{j_1},\, c_{j_1})}$.
Since both $|z_0-z_1|<d$ and $|c_0-c_{j_1}| < d$, it must be that $z_1\neq
c_{j_1}$. Therefore $K_{|z-c_{j_1}|^2\varphi}(z_1, w_0)=0$.
\end{proof}

When $c\not\in\Omega$, then $K_{|z-c|^2\varphi}(z, w)=\frac{K_\varphi(z,\,
w)}{(z-c)(\overline{w}-\overline{c})}$ by Theorem~\ref{prelim:gTheorem}, so the
zero set of $K_{|z-c|^2\varphi}(z, w)$ corresponds to the zero set of
$K_\varphi(z, w)$ in that case. An interpretation of
Theorem~\ref{zeros:ZerosTheoremD} is that for $c\in\Omega$ as $c$ approaches
the boundary of $\Omega$, the zero set of $K_{|z-c|^2\varphi}(z, w)$ approaches
the zero set of $K_{\varphi}(z, w)$. The following corollary to
Theorem~\ref{zeros:ZerosTheoremB} does not assume that $c$ is near the boundary
of $\Omega$, though unlike in Theorem~\ref{zeros:ZerosTheoremD} we assume $c$
is adapted to a zero of the kernel.

\begin{corollary}[Corollary to Theorem~\ref{zeros:ZerosTheoremB}]\label{zeros:CorollaryToTheoremB}
Let $\Omega$ be a domain in $\C$, and let $\varphi$ be a weight on $\Omega$.
Suppose $c, w_0 \in \Omega$ such that $K_\varphi(z, w_0)$ has a zero of order $m>1$ at
$z=c$. Then there exist $z_1, z_2, \dotsc, z_{m-1}, w_1 \in \Omega$ with the
$z_j$ near $z_0$ and $w_1$ near $w_0$ such that $K_{|z-c|^2\varphi}(z_j, w_1)=0$
for $j=1,\dotsc, m-1$.
\end{corollary}

\begin{proof}
Apply Hurwitz's Theorem to the conclusion of
Theorem~\ref{zeros:ZerosTheoremB}.
\end{proof}
\begin{theorem}\label{zeros:ZerosTheoremF}
\hspace{0.2cm}
\begin{enumerate}
  \item[A.] Suppose $\Omega\subset\C$ is a domain, and $\varphi$ a weight, and
  $\{c_j\}_{j=1}^\infty$ is a sequence in $\Omega$ converging to a point
  $c_0\in \partial \Omega$ such that for fixed $z$, $\frac{K_\varphi(z,
  c_j)}{\sqrt{K_\varphi(c_j, c_j)}} \to 0$ as $j\to \infty$. Suppose
  also that $K_{|z-c|^2\varphi}(z_0, w_0)=0$ for all $c\in\Omega$. Then either
	\begin{enumerate}
	  \item both $K_\varphi(z_0, w)\equiv 0$ and $K_{|z-c|^2\varphi}(z_0, w)\equiv
	  0$ as functions of $w$ for all $c$; or
	  \item both $K_\varphi(z, w_0)\equiv 0$ and $K_{|z-c|^2\varphi}(z, w_0)\equiv
	  0$ as functions of $z$ for all $c$.
	\end{enumerate}
  \item[B.] For any domain $\Omega$ and weight $\varphi$, if $K_\varphi(z,
  w_0)\equiv 0$ as a function of $z$, then for all $c\in\C$,
  $K_{|z-c|^2\varphi}(z, w_0)\equiv 0$ as well.
\end{enumerate}
\end{theorem}

\begin{rem}
Part (B) is similar to Theorem~\ref{zeros:ZerosTheoremA} and follows from
Theorem~\ref{zeros:ZerosTheoremA}, the hypothesis that $K_\varphi(z, w_0)\equiv
0$, and continuity.
\end{rem}

\begin{proof} We prove part (A) first. The proof of part (B) will be obvious
from the proof of part (A) and is omitted.

Let $c\in\Omega$. Assume first that $z_0\neq c$ and $w_0 \neq c$. Then by
Theorem~\ref{decomp:ZeroReducingTheorem} we must have
\begin{equation}
K_\varphi(z_0, w_0) = \frac{K_\varphi(z_0, c)K_\varphi(c, w_0)}{K_\varphi(c,
c)}. \label{zeros:thing}
\end{equation}

The right hand side of Equation~\ref{zeros:thing} vanishes when we replace $c$
with $c_j$ and let $j\to\infty$. Hence $K_\varphi(z_0, w_0)=0$, and therefore
either $K_\varphi(z_0, c)=0$ or $K_\varphi(c, w_0)=0$. One of these two
conditions must hold for a set of values of $c$ having an accumulation point,
hence for all $c$. Assume without loss of generality that $K_\varphi(c, w_0)=0$
for all $c$. Thus $K_\varphi(z, w_0)\equiv 0$ as a function of $z$. But then
\[
K_\varphi(z, w_0) = \frac{K_\varphi(z, c)K_\varphi(c, w_0)}{K_\varphi(c,
c)} = 0 \;\;\text{for all } z,
\]
and hence (by Theorem~\ref{decomp:ZeroReducingTheorem}) $K_{|z-c|^2\varphi}(z,
w_0)\equiv 0$ as a function of $z$.
\end{proof}

Since Theorems \ref{zeros:ZerosTheoremD} and \ref{zeros:ZerosTheoremF} have a
hypothesis requiring or implied by the condition 
\[
\frac{K_\varphi(z,
c)}{\sqrt{K_\varphi(c, c)}}\to 0 \text{ as }c\to c_0\in\partial \Omega,
\] 
we state sufficient conditions on a domain for this limit condition to be
satisfied. Below is ~\cite[Lemma~4.1 part~2]{fu_compactness_1998} which is
``implicit in work of Pflug (see~\cite[Section~7.6]{jarnicki_invariant_2005})
and Ohsawa~\cite{ohsawa_remark_1981} on the completeness of the Bergman metric''
according to Fu and Straube~\cite{fu_compactness_1998}.

\begin{theorem}\label{zeros:outercone}
Let $\Omega\subset\C^n$ be a bounded pseudoconvex domain.
Suppose $p_0$ is a point in the boundary of $\Omega$ satisfying the following
outer cone condition:
\begin{quote}
there exist $r\in(0,1]$, $a\geq 1$, and a sequence $\{w_\ell\}_{\ell=1}^\infty$
of points $w_\ell \not\in \Omega$ with $\lim_{\ell\to\infty} w_\ell = p_0$ and
$\Omega\cap B(w_\ell, r\norm{w_\ell - p_0}^a)=\emptyset$.
\end{quote}
Then for any sequence $\{p_j\}_{j=1}^\infty\subset\Omega$ converging to $p_0$,
\[
\lim_{j\to\infty} \frac{K^\Omega(z, p_j)}{\sqrt{K^\Omega(p_j, p_j)}} =0.
\]
\end{theorem}

The outer cone condition of Theorem~\ref{zeros:outercone} is satisfied when
$\Omega$ has $C^1$ boundary, for example. Pseudoconvexity is a central notion in
several complex variables which reduces to a triviality for domains of a single
complex dimension: every domain in the plane is
pseudoconvex~\cite{krantz_function_2001}. Because we wish to also have the
conclusion of the above theorem for certain weighted kernels, we show that the
property addressed by the theorem is preserved when the weight of a kernel is
multiplied by the modulus squared of a linear factor.

\begin{theorem}\label{zeros:outerconeweighted}
Suppose $\Omega\subset\C$ is a domain, $p_0\in\partial\Omega$, and
$\{p_j\}_{j=1}^\infty\subset\Omega$ is a sequence with $p_j\to p_0$ as
$j\to\infty$ such that $\frac{K_\varphi(z, p_j)}{\sqrt{K_\varphi(p_j, p_j)}}\to
0$ as $j\to\infty$ locally uniformly. Then for any $c\in\Omega$ with
$K_\varphi(c, c)\neq 0$, $\frac{K_{|z-c|^2\varphi}(z,
p_j)}{\sqrt{K_{|z-c|^2\varphi}(p_j, p_j)}}\to 0$ as $j\to\infty$ locally uniformly.
\end{theorem}

\begin{proof}
From Theorem~\ref{decomp:ZeroReducingTheorem} we get
\begin{align*}
&\frac{K_{|z-c|^2\varphi}(z,p_j)}{\sqrt{K_{|z-c|^2\varphi}(p_j,p_j)}}
=\frac{
	\frac{K_\varphi(z,\, p_j)K_\varphi(c,\, c) - K_\varphi(z,\, c)K_\varphi(c,\,
	p_j)}{ (z-c)(\overline{p_j}-\overline{c})K_\varphi(c,\,c)}
}{
	\left( \frac{K_\varphi(p_j,\, p_j)K_\varphi(c,\, c) - |K_\varphi(p_j,\, c)|^2}{
	|p_j-c|^2K_\varphi(c,\,c)} \right)^{1/2}
}\\\\
&\qquad =\frac{
	|p_j-c|^2K_\varphi(c, c)^{1/2}
}{
	(z-c)(\overline{p_j}-\overline{c})K_\varphi(c, c)
}
\cdot
\frac{
	K_\varphi(z, p_j)K_\varphi(c, c) - K_\varphi(z, c)K_\varphi(c, p_j)
}{
	\left( K_\varphi(p_j, p_j)K_\varphi(c, c) - |K_\varphi(p_j, c)|^2 \right)^{1/2}
}\\\\
&\qquad =\frac{
	(p_j - c)
}{
	(z-c)K_\varphi(c, c)^{1/2}
}
\cdot
\frac{
	\frac{K_\varphi(z,\, p_j)K_\varphi(c,\, c)}{K_\varphi(p_j,\, p_j)^{1/2}}
	-\frac{K_\varphi(z,\, c)K_\varphi(c,\, p_j)}{K_\varphi(p_j,\, p_j)^{1/2}}
}{
	\left( K_\varphi(c, c) - \frac{|K_\varphi(p_j,\, c)|^2}{K_\varphi(p_j,\, p_j)}
	\right)^{1/2} }.
\end{align*}
The first factor approaches a constant as $j\to\infty$. In the second factor,
every fraction in the numerator and the denominator approaches zero as
$j\to\infty$ locally uniformly by hypothesis, so the second factor approaches
zero as $j\to\infty$ locally uniformly. This proves the theorem.
\end{proof}

\section{Further questions}

Consider the (unweighted) kernel $K(z, w)$ for the unit disk $\D$. By summing an
appropriate orthonormal basis in Equation~(2), it can be shown that for any real
$\alpha$ greater than $-2$,
\begin{equation*} 
K_{|z|^\alpha}(z, w) = K(z, w) + \frac{\alpha}{2\pi (1-z\overline{w})} 
= \left(1+\frac{\alpha}{2} - \frac{\alpha}{2}z\overline{w}\right)K(z, w).
\end{equation*}
(The reader might verify that this formula agrees with
Theorem~\ref{decomp:ZeroReducingTheorem} when $\alpha = 2p$, $p\in\N$.) Now
let $c\in D$ and $p\in\N$. Using this formula, the classical change of variables
theorem for Bergman kernels, and Theorem~\ref{prelim:gTheorem}, one obtains
\begin{equation*}
\begin{split}
K_{|z-c|^{2p}}(z, w) &= \frac{K_{|\mu_c|^{2p}}(z,
w)}{(1-\overline{c}z)^p(1-c\overline{w})^p} \\ 
&= ((p+1)-p\mu_c(z)\overline{\mu_c(w)})\frac{K(z,
w)}{(1-\overline{c}z)^p(1-c\overline{w})^p}.
\end{split} 
\end{equation*} 
What is the formula if $p$ is allowed to be real? That is, what is the formula
for $K_{|z-c|^{\alpha}}(z, w)$, $\alpha\in\R$? In particular, what is the
formula when $\alpha = 1$?

Generalizing the previous question, is there a technique for computing
$K_{\varphi}^\Omega(z, w)$ explicitly in terms of $K^\Omega(z, w)$ in the case
that $\varphi$ is the modulus of a meromorphic function rather than the
\emph{square} of the modulus a meromorphic function? Is there such a technique
when $\varphi$ is harmonic?

\bibliographystyle{amsplain}
\bibliography{bibliography}

\end{document}